\documentclass[12pt]{amsart}

\usepackage[mathscr]{eucal}
\usepackage{amsmath,amssymb,amscd,amsthm}%,latexsym
\usepackage[T1]{fontenc}
\usepackage{color}
\usepackage{amsmath,amsthm,amssymb,graphicx}
\usepackage{epsfig}
\newtheorem{lemma}{Lemma}
\newtheorem{theorem}{Theorem}

\theoremstyle{definition}

\theoremstyle{plane}

\def \beq{ \begin{equation} }
\def \eeq{\end{equation}}

\usepackage[top=4cm, bottom=3.8cm, left=3.5cm, right=3.5cm]{geometry}
\title{The continuous transition of Hamiltonian vector fields through manifolds of constant curvature}
\begin{document}
\maketitle
\markboth{Florin Diacu, Slim Ibrahim, and J\textogonekcentered{e}drzej \'Sniatycki}{The continuous transition of vector fields through manifolds of constant curvature}
\author{\begin{center}
{\bf Florin Diacu,  Slim Ibrahim, and J\textogonekcentered{e}drzej \'Sniatycki}\\
\smallskip
{\footnotesize Pacific Institute for the Mathematical Sciences\\
and\\
Department of Mathematics and Statistics\\
University of Victoria\\
P.O.~Box 1700 STN CSC\\
Victoria, BC, Canada, V8W 2Y2\\
diacu@uvic.ca, ibrahims@uvic.ca, sniatyck@ucalgary.ca\\
}\end{center}

}

\begin{center}
\today
\end{center}

\begin{abstract}
We ask whether Hamiltonian vector fields defined on spaces of constant Gaussian curvature $\kappa$ (spheres, for $\kappa>0$, and hyperbolic spheres, for $\kappa<0$), pass continuously through the value $\kappa=0$ if the potential functions $U_\kappa, \kappa\in\mathbb R$, that define them satisfy the property $\lim_{\kappa\to 0}U_\kappa=U_0$, where $U_0$ corresponds to the Euclidean case. We prove that the answer to this question is positive, both in the 2- and 3-dimensional cases, which are of physical interest, and then apply our conclusions to the gravitational $N$-body problem.
\end{abstract}

%%%%%%%%%%%%
%%%%%%%%%%%%
%%%%%%%%%%%%
\section{Introduction}
%%%%%%%%%%%%
%%%%%%%%%%%%
%%%%%%%%%%%%

The attempts to extend classical equations of mathematical
physics from Euclidean space to more general manifolds is
not new. This challenging trend started in the 1830s with the work of J\'anos Bolyai and Nikolai Lobachevsky, who tried to generalize the 2-body problem of celestial mechanics to hyperbolic space, \cite{Bol}, \cite{Lob}. This particular topic is much researched today in the context of the $N$-body problem in spaces of constant curvature, \cite{Diacu01}, \cite{Diacu02}, \cite{Diacu03}, \cite{Diacu04}, \cite{DK}, \cite{DPR}, \cite{DT}. But extensions to various manifolds have also been pursued for PDEs, even beyond the boundaries of classical mechanics, such as for Schr\"odinger's equation \cite{Banica}, \cite{Burq}, \cite{Ionescu}, and the Vlasov-Poisson system, \cite{DILS}, to mention just a couple.

In all these problems it is important to understand whether the extension of the equations has any physical sense. This issue, however, is specific to each case and cannot be treated globally. The natural generalization of the gravitational $N$-body problem to spheres and hyperbolic spheres, for instance, is justified through properties preserved by the Kepler potential, which is, for any constant value of the curvature, a harmonic function that keeps all bounded orbits closed, \cite{Diacu03}, \cite{Diacu04}. Other problems need different unifying justifications.

But a basic aspect of common interest to all attempts of extending the dynamics from Euclidean space to more general manifolds is that the classical equations are recovered when the manifolds are flattened. In the absence of this property, any extension is devoid of meaning. Of course, this attribute alone is not enough to warrant a specific generalization, since many extensions can manifest this quality, so a choice must be also based on other criteria. Nevertheless, checking the presence of this feature is an indispensable first step towards justifying any extension of the classical equations to more general spaces.

In this note we've set a modest goal in the basic direction mentioned above. We would like to see whether Hamiltonian vector fields defined on spheres and hyperbolic spheres tend to the classical Euclidean vector field as the curvature approaches zero, assuming that the classical potential is reached in the limit in a sense that will be made precise. As we will see, the answer to this problem is positive, which may be surprising, given the fact that, from the geometrical point of view, 2-spheres do not project isometrically on the Euclidean plane, and the same is true for 3-spheres relative to the Euclidean space. The subtle point of the results we obtain stays with how the limit is taken. 

To formalize this problem, let us consider a family of Hamiltonian vector fields, given by potential functions $U_\kappa$, $\kappa\in \mathbb R$, defined on spheres of Gaussian curvature $\kappa>0$,
$$
\mathbb S_\kappa^2=\{(x,y,z) \ \! | \ \! x^2+y^2+z^2=\kappa^{-1}\},
$$
embedded in the Euclidean space $\mathbb R^3$, as well as hyperbolic spheres of curvature $\kappa<0$, 
$$
\mathbb H_\kappa^2=\{(x,y,z) \ \! | \ \! x^2+y^2-z^2=\kappa^{-1}\},
$$
embedded in the Minkowski space $\mathbb R^{2,1}$, such that $U_\kappa$ tends to a potential $U_0$ of the Euclidean plane as $\kappa\to 0$ while the Euclidean distances in $\mathbb R^3$ and the Minkowski distances in $\mathbb R^{2,1}$ are kept fixed when $\kappa\to 0$. This way of approaching the limit is essential here, since allowing distances to vary with the curvature makes all the points of the sphere and the hyperbolic sphere run to infinity as $\kappa\to 0$.

We would like to know whether the equations of motion defined for $\kappa\ne 0$ tend to the classical equations of motion defined in flat space, i.e.\ for $\kappa=0$. We are also interested in answering the same question in the 3-dimensional case, in other words to decide what happens when the manifolds $\mathbb S_\kappa^2$ and
$\mathbb H_\kappa^2$ are replaced by 3-spheres,
$$
\mathbb S_\kappa^3=\{(x,y,z,w) \ \! | \ \! x^2+y^2+z^2+w^2=\kappa^{-1}\}
$$
(embedded in $\mathbb R^4$), and hyperbolic 3-spheres,
$$
\mathbb H_\kappa^3=\{(x,y,z,w) \ \! | \ \! x^2+y^2+z^2-w^2=\kappa^{-1}\}
$$
(embedded in the Minkowski space $\mathbb R^{3,1}$), respectively. Of course, we could work in general for any finite space dimension, but the cases of main physical interest are dimensions 2 and 3. Moreover, we would like to explicitly see what the equations of motion look like in these two particular cases, such that we can use them in the gravitational $N$-body problem. Another reason for including both the 2- and 3-dimensional case in this note, instead of treating only the latter, is that the former allows the reader to build intuition and  make then an easy step to the higher dimension.

To address the above issues for $N$ point masses, let us initially consider the more general problem in which the bodies are interacting on a complete, connected, $n$-dimensional Riemannian manifold $\mathcal M$ ($n$ being a fixed positive integer), under a law given by a potential function. So let $N$ point particles (bodies) of masses $m_1,\dots, m_N>0$ move on the manifold $\mathcal M$. In some suitable coordinate system, the position and velocity of the body $m_r$ are described by the vectors 
$$
{\bf x}_r=(x_1^r,\dots,x_n^r), \ \ \dot{\bf x}_r=(\dot x_1^r,\dots,\dot x_n^r),\ \ r=\overline{1,N},
$$
respectively. We attach to every particle $m_r$ a metric tensor given by the matrix $G_r=(g_{ij}^r)$, and its inverse, $G_r^{-1}=(g_r^{ij})$, at the point of the manifold $\mathcal M$ where the particle $m_r, r=\overline{1,N},$ happens to occur at a given time instant.

The law of motion is given by a sufficiently smooth potential function, 
$$
U\colon{\mathcal M}^N\setminus\Delta\to(0,\infty),\ \ U=U({\bf x}),
$$ where ${\bf x}=({\bf x}_1,\dots, {\bf x}_N)$ is 
the configuration of the particle system and $\Delta$
represents the set of singular configurations, i.e.\ positions of the bodies for which the potential is not defined. Then $U$ generates a Lagrangian function
\begin{equation}
L({\bf x}, \dot{\bf x})={1\over 2}\sum_{r=1}^N\sum_{i,j=1}^nm_rg_{ij}^r\dot{x}_i^r\dot{x}_j^r
-U({\bf x}),
\label{lag}
\end{equation}
which corresponds to a Hamiltonian vector field.
In this setting, we can derive the equations of motion on the Riemannian manifold $\mathcal M$ in the result stated below, whose proof we also provide for the completeness
of our presentation.

%%%%%%%% LEMMA 1
\begin{lemma}
%%%%%%%%
Assume that the point masses $m_1,\dots, m_N>0$ interact on the complete and connected $n$-dimensional Riemannian manifold $\mathcal M$ under the law imposed by Lagrangian \eqref{lag}. Then the system of differential equations describing the motion of these particles
has the form
\begin{equation}\label{motion}
m_r\ddot{x}_s^r=-\sum_{i=1}^ng_r^{si}\frac{\partial U}{\partial{x}_i^r}
-m_r\sum_{l,j=1}^n\Gamma_{lj}^{s,r}\dot{x}_l^r\dot{x}_j^r, \ s=\overline{1,n},\
r=\overline{1,N},
\end{equation}
where 
\begin{equation}\label{Christoffel}
\Gamma_{lj}^{s,r}=\frac{1}{2}\sum_{i=1}^ng_r^{si}\left(\frac{\partial g_{il}^r}{\partial{x}_j^r}+\frac{\partial g_{ij}^r}{\partial x_l^r}-\frac{\partial g_{lj}^r}{\partial x_i^r}\right), \ s=\overline{1,n},
\end{equation}
are the Christoffel symbols corresponding to the particle $m_r,\  r=\overline{1,N}$.
\end{lemma}
%%%%%%%%%
\begin{proof}
Recall first that $g_{ij}^r=g_{ji}^r$ and $g^{ij}_r=g^{ji}_r$ for $i,j=\overline{1,n}$. A simple computation leads to
$$
\frac{\partial L}{\partial x_i^r}=\frac{1}{2}\sum_{l,j=1}^nm_r\frac{\partial g_{lj}^r}{\partial x_i^r}\dot{x}_l^r\dot{x}_j^r-\frac{\partial U}{\partial x_i^r}, \ \ \ \
\frac{\partial L}{\partial\dot{x}_i^r}=\sum_{l=1}^nm_rg_{il}^r\dot{x}_l^r,
$$
$$
\frac{d}{dt}\left(\frac{\partial L}{\partial\dot{x}_i^r} \right)=\sum_{l=1}^nm_rg_{il}^r
\ddot{x}_l^r+\sum_{l,j=1}^nm_r\frac{\partial g_{ij}^r}{\partial x_l^r}\dot{x}_l^r\dot{x}_j^r.
$$
Then the Euler-Lagrange equations,
$$
\frac{d}{dt}\left(\frac{\partial L}{\partial\dot{x}_i^r} \right)=\frac{\partial L}{\partial x_i^r},
\  \ i=\overline{1,n},\ \ r=\overline{1,N},
$$
which describe the motion of the $N$ particles, take the form
$$
m_r\sum_{l=1}^ng_{il}^r
\ddot{x}_l^r+m_r\sum_{l,j=1}^n\left(\frac{\partial g_{ij}^r}{\partial x_l^r}
-\frac{1}{2}\frac{\partial g_{lj}^r}{\partial x_i^r}\right)\dot{x}_l^r\dot{x}_j^r=-\frac{\partial U}{\partial x_i^r}, \  i=\overline{1,n},\ r=\overline{1,N}.
$$
Let us fix $r$ in the above equations, multiply the $i$th equation by $g_r^{si}$,
and add all $n$ equations thus obtained. The result is 
$$
m_r\sum_{i,l=1}^ng_r^{si}g_{il}^r
\ddot{x}_l^r+m_r\sum_{i,l,j=1}^ng_r^{si}\left(\frac{\partial g_{ij}^r}{\partial x_l^r}
-\frac{1}{2}\frac{\partial g_{lj}^r}{\partial x_i^r}\right)\dot{x}_l^r\dot{x}_j^r=
-\sum_{i=1}^ng_r^{si}\frac{\partial U}{\partial x_i^r}
, \ r=\overline{1,N}.
$$ 
Using the fact that $\sum_{i=1}^ng_r^{si}g_{il}^r=\delta_{sl}$,
where $\delta_{sl}$ is the Kronecker delta, and the identity
$$
\sum_{j,l=1}^n\frac{\partial g_{ij}^r}{\partial x_l^r}\dot{x}_l^r\dot{x}_j^r=
\frac{1}{2}\sum_{j,l=1}^n\left(\frac{\partial g_{il}^r}{\partial x_j^r}+
\frac{\partial g_{ij}^r}{\partial x_l^r}\right)\dot{x}_l^r\dot{x}_j^r,
$$
which is easy to prove by expanding the double sums, the above equations become those given in the statement of the lemma, a remark that completes the proof.
\end{proof}

%%%%%%%%%%%%
%%%%%%%%%%%%
%%%%%%%%%%%%
\section{The 2-dimensional case}
%%%%%%%%%%%%
%%%%%%%%%%%%
%%%%%%%%%%%%

In this section we will prove the continuity of Hamiltonian vector fields through the value $\kappa=0$ of the curvature, i.e.\ when we move from   $\mathbb S_\kappa^2$ to $\mathbb H_\kappa^2$ through $\mathbb R^2$. For this purpose, let us first define some trigonometric functions that unify circular and hyperbolic trigonometry, namely the $\kappa$-sine function, ${\rm sn}_\kappa$, as
$$
{\rm sn}_{\kappa}s:=\left\{
\begin{array}{rl}
{\kappa}^{-1/2}\sin{\kappa}^{1/2}s & {\rm if }\ \ \kappa>0\\
s & {\rm if }\ \ \kappa=0\\
|{\kappa}|^{-{1/2}}\sinh |{\kappa}|^{1/2}s & {\rm if }\ \ \kappa<0,
\end{array}  \right.
$$
the $\kappa$-cosine function, ${\rm csn}_\kappa$, as
$$
{\rm csn}_{\kappa}s:=\left\{
\begin{array}{rl}
\cos{\kappa}^{1/2}s & {\rm if }\ \ \kappa>0\\
1 & {\rm if }\ \ \kappa=0\\
\cosh |{\kappa}|^{1/2}s & {\rm if }\ \ \kappa<0,
\end{array}  \right.
$$
as well as the functions $\kappa$-tangent, ${\rm tn}_\kappa$, and $\kappa$-cotangent,
${\rm ctn}_\kappa$, as 
$${\rm tn}_{\kappa}s:={{\rm sn}_{\kappa}s\over {\rm csn}_{\kappa}s}\ \ \ {\rm and}\ \ \
{\rm ctn}_{\kappa}s:={{\rm csn}_{\kappa}s\over {\rm sn}_{\kappa}s},$$
respectively. The following relationships, which will be useful in subsequent computations, follow from the above definitions:
$$
{\kappa}\ \! {\rm sn}_{\kappa}^2s+{\rm csn}_{\kappa}^2s=1,
$$
$$
\frac{d}{ds}{\rm sn}_\kappa s={\rm csn}_\kappa s\ \ \ {\rm and}\ \ \ \frac{d}{ds}{\rm csn}_\kappa s=-\kappa \ \!{\rm sn}_\kappa s.
$$
Also notice that all the above unified trigonometric functions are continuous relative to $\kappa$.

We can now prove the following result.

%%%%%%%% THEOREM 1
\begin{theorem}
%%%%%%%%
Consider the point masses (bodies) $m_1,\dots,m_N$ on the manifold
$\mathbb M_\kappa^2$ (representing $\mathbb S_\kappa^2$ or $\mathbb H_\kappa^2$), $\kappa\ne 0$, whose positions are given by spherical coordinates $(s_r,\varphi_r), r=\overline{1,N}$. Then the equations of motion for these bodies are
%%%%%%%
\begin{equation}\label{motion2}
\begin{cases}
\ddot s_r=-\dfrac{\partial U_\kappa}{\partial s_r}+\dot\varphi_r^2\ \!{\rm sn}_\kappa s_r \ \!{\rm csn}_\kappa s_r\cr
\ddot\varphi_r=-{\rm sn}_\kappa^{-2}s_r\ \!\dfrac{\partial U_\kappa}{\partial\varphi_r}-2 \dot s_r\dot\varphi_r\ \!{{\rm ctn}_\kappa}s_r, \ \ r=\overline{1,N},
\end{cases}
\end{equation}
%%%%%%%
where
$$
U_\kappa\colon(\mathbb M_\kappa^2)^N\setminus\Delta_\kappa\to(0,\infty),\ \ \kappa\ne 0, 
$$
represent the potentials, with the sets $\Delta_\kappa,\ \kappa\in\mathbb R,$ corresponding to singular configurations. Moreover, if
$
\lim_{\kappa\to 0}U_\kappa=U_0
$
while the Euclidean distances in $\mathbb R^3$ (and the Minkowski distances in $\mathbb R^{2,1}$) between the North Pole $N=(0,0,0)$ and the bodies are kept fixed, where
$$
U_0\colon(\mathbb R^2)^N\setminus\Delta_0\to(0,\infty)
$$ 
is the potential in the Euclidean case and $\Delta_0$ is the
corresponding set of singular configurations, then system \eqref{motion2} tends to the classical equations
\begin{equation}\label{Newton2}
\ddot x_r=-\dfrac{\partial U_0}{\partial x_r},\ \
\ddot y_r=-\dfrac{\partial U_0}{\partial y_r}, \ \ r=\overline{1,N}.
\end{equation}
\end{theorem}
%%%%%%%%%%% PROOF
\begin{proof}
In $\mathbb M_\kappa^2$, the equations of motion \eqref{motion} for the $N$-body system take the form 
\begin{equation}\label{one-body}
\small
\begin{cases}
\ddot{x}_1^r=-g^{11}_r\dfrac{\partial U_\kappa}{\partial x_1^r}-g^{12}_r\dfrac{\partial U_\kappa}{\partial x_2^r}-\Gamma_{11}^{1,r}(\dot x_1^r)^2-(\Gamma_{12}^{1,r}+\Gamma_{21}^{1,r})\dot x_1^r\dot x_2^r-\Gamma_{22}^{1,r}(\dot x_2^r)^2\cr
\vspace{-0.35cm}\cr
\ddot{x}_2^r=-g^{21}_r\dfrac{\partial U_\kappa}{\partial x_1^r}-g^{22}_r\dfrac{\partial U_\kappa}{\partial x_2^r}-\Gamma_{11}^{2,r}(\dot x_1^r)^2-(\Gamma_{12}^{2,r}+\Gamma_{21}^{2,r})\dot x_1^r\dot x_2^r-\Gamma_{22}^{2,r}(\dot x_2^r)^2, 
\end{cases}
\end{equation}
$ r=\overline{1,N}.$ Consider now the $(s,\varphi)$-coordinates on $\mathbb M_\kappa^2$, i.e.\ take in system \eqref{one-body} the variables $$
s_r:=x_1^r, \ \ \varphi_r:=x_2^r,\ \ r=\overline{1,N},
$$
given in terms of extrinsic $(x,y,z)$-coordinates  by
$$
x_r={\rm sn}_\kappa s_r\cos\varphi_r,\ \ 
y_r={\rm sn}_\kappa s_r\sin\varphi_r,\ \ 
z_r=|\kappa|^{-1/2}{\rm csn}_\kappa s_r-|\kappa|^{-1/2},
r=\overline{1,N}.
$$
Notice that for $\kappa\to 0$, we have $z\to 0$ and the remaining relations give the polar coordinates in $\mathbb R^2$.

To make precise how we take the above limit, let us remark that while keeping the bodies fixed as $\kappa\to 0$, the geodesic distances $s_r, r=\overline{1,N},$ are not fixed. Indeed, we can write $s_r=|\kappa|^{-1/2}\alpha_r$, where
$\alpha_r$ is the angle from the centre of the sphere that subtends the arc of geodesic length $s_r$. The Euclidean/Minkowski distance corresponding to this arc is given by $\tau_r:=2\ \!{\rm sn}_\kappa(\alpha_r/2)$, so we can conclude that $s_r=2\ \!{\rm sn}_\kappa^{-1}(\tau_r/2)$. The quantity $\tau_r$ is assumed fixed as $\kappa$ varies, and
$s_r\to\tau_r$ as $\kappa\to 0$.

Differentiating the expressions given the change of coordinates we obtain
$$
dx_r={\rm csn}_\kappa s_r\cos\varphi_r\ \! ds-{\rm sn}_\kappa s_r\sin\varphi_r\ \! d\varphi_r,
$$
$$
dy_r={\rm csn}_\kappa s_r\sin\varphi_r\ \! ds_r +
{\rm sn}_\kappa s_r\cos\varphi_r\ \! d\varphi_r,
$$
$$
dz_r=-\sigma|\kappa|^{1/2}{\rm sn}_\kappa s_r\ \! ds_r,
$$
where $\sigma=1$ for $\kappa>0$, but $\sigma=-1$ for $\kappa<0$.
Using these expression of the differentials, we can compute the line element $dx_r^2+dy_r^2+\sigma dz_r^2$ and obtain that the metric tensor and its inverse are given by matrices of the form
$$
G_r=(g_{ij}^r)=\begin{pmatrix}
1 & 0\\
0 & {\rm sn}_\kappa^2s_r
\end{pmatrix},
\   \ \ 
G_r^{-1}=(g_r^{ij})=\begin{pmatrix}
1 & 0\\
0 & \frac{1}{{\rm sn}_\kappa^2s_r}
\end{pmatrix},
$$
respectively, and that they cover the entire spectrum of metrics for $\kappa\in\mathbb R$. So we obtain that
$$
g_{11}^r=1,\ g_{12}^r=g_{21}^r=0,\ g_{22}^r={\rm sn}_\kappa^2s_r,\
g^{11}_r=1,\ g^{12}_r=g^{21}_r=0,\ g^{22}_r=\frac{1}{{\rm sn}_\kappa^2s_r}.
 $$
Using the matrices $G$ and $G^{-1}$ as well as equations \eqref{Christoffel}, we compute the Christoffel symbols and obtain 
$$
\Gamma_{11}^{1,r}=\Gamma_{12}^{1,r}=\Gamma_{21}^{1,r}=\Gamma_{11}^{2,r}=\Gamma_{22}^{2,r}=0,
$$
$$
\Gamma_{22}^{1,r}=-{\rm sn}_\kappa s_r\ \!{\rm csn}_\kappa s_r,\ \
\Gamma_{12}^{2,r}=\Gamma_{21}^{2,r}={\rm ctn}_\kappa s_r.
$$
Using the above results and reintroducing the index $r$ to point out the dependence of the equations of motion on the position of each body, system \eqref{one-body} becomes
%%%%%%%
\begin{equation}
\begin{cases}
\ddot s_r=-\dfrac{\partial U_\kappa}{\partial s_r}+\dot\varphi_r^2\ \!{\rm sn}_\kappa s_r \ \!{\rm csn}_\kappa s_r\cr
\ddot\varphi_r=-{\rm sn}_\kappa^{-2}s_r\ \!\dfrac{\partial U_\kappa}{\partial\varphi_r}-2 \dot s_r\dot\varphi_r\ \!{{\rm ctn}_\kappa}s_r, \ \ r=\overline{1,N}.
\end{cases}
\end{equation}
%%%%%%%
Given the definitions of the unified trigonometric functions at $\kappa=0$, the above system takes in Euclidean space the form
\begin{equation}\label{polar2}
\begin{cases}
\ddot s_r=-\dfrac{\partial U_0}{\partial s_r}+s_r\dot\varphi_r^2\cr
\ddot\varphi_r=-s_r^{-2}\ \!\dfrac{\partial U_0}{\partial\varphi_r}-2s_r^{-1}\dot s_r\dot\varphi_r, \ \ r=\overline{1,N}.
\end{cases}
\end{equation}
But, as noted above, the $(s,\varphi)$-coordinates of $\mathbb S_\kappa^2$ and $\mathbb H_\kappa^2$ (for $\kappa\ne 0$), become polar coordinates in $\mathbb R^2$ (for $\kappa=0$), so if we write
$$
x_r=s_r\cos\varphi_r, \ \ y_r=s_r\sin\varphi_r,\ \ r=\overline{1,N},
$$
and perform the computations, system \eqref{polar2} takes the desired form \eqref{Newton2}. This remark completes the proof.
\end{proof}

%%%%%%%%%%%%
%%%%%%%%%%%%
%%%%%%%%%%%%
\section{The 3-dimensional case}
%%%%%%%%%%%%
%%%%%%%%%%%%
%%%%%%%%%%%%

In this section we consider the motion in $\mathbb M_\kappa^3$, which stands for  $\mathbb S_\kappa^3$ or $\mathbb H_\kappa^3$. Our main goal is to prove the following result.

%%%%%%%% THEOREM 2
\begin{theorem}
%%%%%%%%
Consider the point masses (bodies) $m_1,\dots,m_N$ on the manifold
$\mathbb M_\kappa^3$ (representing $\mathbb S_\kappa^3$ or $\mathbb H_\kappa^3$), $\kappa\ne 0$, whose positions are given in hyperspherical coordinates $(s_r,\varphi_r,\theta_r), r=\overline{1,N}$. Then the equations of motion for these bodies are given by the system
%%%%%%%%%%
\begin{equation}\label{motion3}
\begin{cases}
\ddot s_r=-\dfrac{\partial U_\kappa}{\partial s_r}+
(\dot\varphi_r^2\ \!+\dot\theta_r^2\sin^2\varphi_r)\ \!{\rm sn}_\kappa s_r\ \!{\rm csn}_\kappa s_r\cr
\vspace{-0.4cm}\cr
\ddot\varphi_r=-{\rm sn}_\kappa^{-2}s_r\ \!\dfrac{\partial U_\kappa}{\partial\varphi_r}+\theta_r^2 \sin\varphi_r\cos\varphi_r
-2\dot s_r\dot\varphi_r\ \!{\rm ctn}_\kappa s_r\cr
\vspace{-0.4cm}\cr
\ddot\theta_r=-{\rm sn}_\kappa^{-2}s_r \sin^{-2}\varphi_r\ \!\dfrac{\partial U_\kappa}{\partial\theta_r}-2\dot s_r\dot\theta_r\ \!{\rm ctn}_\kappa
s_r - 2\dot\varphi_r\dot\theta_r\cot\varphi_r, \ \ r=\overline{1,N},
\end{cases}
\end{equation}
%%%%%%%%%%%
where
$$
U_\kappa\colon(\mathbb M_\kappa^3)^N\setminus\Delta_\kappa\to(0,\infty),\ \ \kappa\ne 0, 
$$
represent the potentials, with the sets $\Delta_\kappa,\ \kappa\in\mathbb R,$ corresponding to singular configurations. Moreover, if
$
\lim_{\kappa\to 0}U_\kappa=U_0
$
while the Euclidean distances in $\mathbb R^4$ (and the Minkowski distances in $\mathbb R^{3,1}$) between the North Pole $N=(0,0,0,0)$ and the bodies are kept fixed, where
$$
U_0\colon(\mathbb R^3)^N\setminus\Delta_0\to(0,\infty)
$$ 
is the potential in the Euclidean case and $\Delta_0$ is the
corresponding set of singular configurations, then system \eqref{motion3} tends to the classical equations,
\begin{equation}\label{Newton3}
\ddot x_r=-\dfrac{\partial U_0}{\partial x_r},\ \
\ddot y_r=-\dfrac{\partial U_0}{\partial y_r}, \ \ 
\ddot z_r=-\dfrac{\partial U_0}{\partial z_r},\ \ r=\overline{1,N}.
\end{equation}
\end{theorem}
%%%%%%%%%%% PROOF
\begin{proof}
In $\mathbb M_\kappa^3$, system \eqref{motion} takes the form
\begin{equation}\label{one-body3}
\begin{cases}
\ddot{x}_1^r=-g^{11}_r\dfrac{\partial U_\kappa}{\partial x_1^r}-g^{12}_r\dfrac{\partial U_\kappa}{\partial x_2^r}
-g^{13}_r\dfrac{\partial U_\kappa}{\partial x_3^r}-\Gamma_{11}^{1,r}(\dot x_1^r)^2-\Gamma_{22}^{1,r}(\dot x_2^r)^2-\Gamma_{33}^{1,r}(\dot x_3^r)^2\cr
{} \ \ \  -(\Gamma_{12}^{1,r}+\Gamma_{21}^{1,r})\dot x_1^r\dot x_2^r-(\Gamma_{13}^{1,r}+\Gamma_{31}^{1,r})\dot x_1^r\dot x_3^r-(\Gamma_{23}^{1,r}+\Gamma_{32}^{1,r})\dot x_2^r\dot x_3^r\cr
\vspace{-0.35cm}\cr
\ddot{x}_2^r=-g^{21}_r\dfrac{\partial U_\kappa}{\partial x_1^r}-g^{22}_r\dfrac{\partial U_\kappa}{\partial x_2^r}
-g^{23}_r\dfrac{\partial U_\kappa}{\partial x_3^r}-\Gamma_{11}^{2,r}(\dot x_1^r)^2-\Gamma_{22}^{2,r}(\dot x_2^r)^2-\Gamma_{33}^{2,r}(\dot x_3^r)^2\cr
{} \ \ \ -(\Gamma_{12}^{2,r}+\Gamma_{21}^{2,r})\dot x_1^r\dot x_2^r-(\Gamma_{13}^{2,r}+\Gamma_{31}^2)\dot x_1^r\dot x_3^r-(\Gamma_{23}^{2,r}+\Gamma_{32}^{2,r})\dot x_2^r\dot x_3^r\cr
\vspace{-0.35cm}\cr
\ddot{x}_3^r=-g^{31}_r\dfrac{\partial U_\kappa}{\partial x_1^r}-g^{32}_r\dfrac{\partial U_\kappa}{\partial x_2^r}
-g^{33}_r\dfrac{\partial U_\kappa}{\partial x_3^r}-\Gamma_{11}^{3,r}(\dot x_1^r)^2-\Gamma_{22}^{3,r}(\dot x_2^r)^2-\Gamma_{33}^{3,r}(\dot x_3^r)^2\cr
{} \ \ \ -(\Gamma_{12}^{3,r}+\Gamma_{21}^{3,r})\dot x_1^r\dot x_2^r-(\Gamma_{13}^{3,r}+\Gamma_{31}^{3,r})\dot x_1^r\dot x_3^r-(\Gamma_{23}^{3,r}+\Gamma_{32}^{3,r})\dot x_2^r\dot x_3^r,  r=\overline{1,N}.
\end{cases}
\end{equation}
%%%%%%%%%
Consider now the intrinsic $(s,\varphi,\theta)$-coordinates on the manifolds $\mathbb M_\kappa^3$, i.e.\ take 
$$
s_r:=x_1^r,\ \  \varphi_r:=x_2^r, \ \ \theta_r:=x_3^r, \ \ r=\overline{1,N},
$$ 
given in terms of the extrinsic $(x,y,z,w)$-coordinates by
\begin{equation}\label{spherical-transf}
\begin{cases}
x_r={\rm sn}_\kappa s_r\sin\varphi_r\sin\theta_r\cr
y_r={\rm sn}_\kappa s_r\sin\varphi_r\cos\theta_r\cr
z_r={\rm sn}_\kappa s_r\cos\varphi_r\cr
w_r=|\kappa|^{-1/2}{\rm csn}_\kappa s_r-|\kappa|^{-1/2}, \ \ r=\overline{1,N}.
\end{cases}
\end{equation}
Notice that for $\kappa\to 0$, we have $w\to 0$ and the remaining relations give the spherical coordinates in $\mathbb R^3$. The remark we made in the proof of Theorem 1 (that the quantities $s_r, r=\overline{1,N}$, vary as we keep the Euclidean/Minkowski distance fixed, but tend to that distance as $\kappa\to 0$) applies here too without any change. 

Differentiating in \eqref{spherical-transf}, we obtain
$$
dx_r={\rm csn}_\kappa s_r\sin\varphi_r\sin\theta_r\ \! ds_r
+{\rm sn}_\kappa s_r\cos\varphi_r\sin\theta_r \ \! d\varphi_r
+{\rm sn}_\kappa s_r\sin\varphi_r\cos\theta_r \ \! d\theta_r,
$$
$$
dy_r={\rm csn}_\kappa s_r\sin\varphi\_rcos\theta_r\ \! ds_r
+{\rm sn}_\kappa s_r\cos\varphi_r\cos\theta_r \ \! d\varphi_r
-{\rm sn}_\kappa s_r\sin\varphi_r\sin\theta_r \ \! d\theta_r,
$$
$$
dz_r={\rm csn}_\kappa s_r\cos\varphi_r\ \! ds_r
-{\rm sn}_\kappa s_r\sin\varphi_r \ \! d\varphi_r
$$
$$
dw_r=-\sigma|\kappa|^{1/2}{\rm sn}_\kappa s_r\ \! ds_r.
$$
Using these expressions we can compute the line element $dx_r^2+dy_r^2+dz_r^2+\sigma dw_r^2$ to obtain the metric tensor and its inverse in matrix form,
$$
G_r=(g_{ij}^r)=\begin{pmatrix}
1 & 0 & 0\\
0 & {\rm sn}_\kappa^2s_r & 0\\
0 & 0 & {\rm sn}_\kappa^2 s_r\sin^2\varphi_r
\end{pmatrix},\ \  
G^{-1}_r=(g^{ij}_r)=\begin{pmatrix}
1 & 0 & 0\\
0 & \frac{1}{{\rm sn}_\kappa^2s_r} & 0\\
0 & 0 & \frac{1}{{\rm sn}_\kappa^2 s_r\sin^2\varphi_r}
\end{pmatrix}.
$$
These matrices give all the metrics for $\kappa\in\mathbb R$. We can thus write that
$$
g_{11}^r=1,\ g_{12}^r=g_{13}^r=g_{21}^r=g_{23}^r=g_{31}^r=g_{32}^r=0,\ g_{22}^r={\rm sn}_\kappa^2 s_r,\ 
g_{33}^r={\rm sn}_\kappa^2 s_r\sin^2\varphi_r,
$$
$$
g^{11}_r=1,\ g^{12}_r=g^{13}_r=g^{21}_r=g^{23}_r=g^{31}_r=g^{32}_r=0,\ g^{22}_r=\frac{1}{{\rm sn}_\kappa^2 s_r},\
g^{33}_r=\frac{1}{{\rm sn}_\kappa^2 s_r\sin^2\varphi_r}.
$$
Then we can compute the Christoffel symbols and obtain
$$
\Gamma_{11}^{1,r}=\Gamma_{11}^{2,r}=\Gamma_{22}^{2,r}=\Gamma_{12}^{1,r}=\Gamma_{21}^{1,r}=
\Gamma_{13}^{1,r}=\Gamma_{31}^{1,r}=
\Gamma_{23}^{1,r}=\Gamma_{32}^{1,r}=0,
$$
$$
\Gamma_{22}^{1,r}=-{\rm sn}_\kappa s_r\ \!{\rm csn}_\kappa s_r,\ \
\Gamma_{33}^{1,r}=-{\rm sn}_\kappa s_r\ \!{\rm csn}_\kappa s_r\ \!
\sin^2\varphi_r,
$$
$$
\Gamma_{13}^{2,r}=\Gamma_{31}^{2,r}=\Gamma_{23}^{2,r}=\Gamma_{32}^{2,r}=\Gamma_{11}^{3,r}=\Gamma_{22}^{3,r}=
\Gamma_{33}^{3,r}=\Gamma_{12}^{3,r}=\Gamma_{21}^{3,r}=0,
$$
$$
\Gamma_{33}^{2,r}=-\sin\varphi_r\cos\varphi_r, 
\Gamma_{12}^{2,r}=\Gamma_{21}^{2,r}=\Gamma_{13}^{3,r}=\Gamma_{31}^{3,r}={\rm ctn}_\kappa s_r, 
\Gamma_{23}^{3,r}=\Gamma_{32}^{3,r}=\cot\varphi_r.
$$
Then system \eqref{one-body3} becomes 
\begin{equation}
\begin{cases}
\ddot s_r=-\dfrac{\partial U_\kappa}{\partial s_r}+
(\dot\varphi_r^2\ \!+\dot\theta_r^2\sin^2\varphi_r)\ \!{\rm sn}_\kappa s_r\ \!{\rm csn}_\kappa s_r\cr
\vspace{-0.4cm}\cr
\ddot\varphi_r=-{\rm sn}_\kappa^{-2}s_r\ \!\dfrac{\partial U_\kappa}{\partial\varphi_r}+\theta_r^2 \sin\varphi_r\cos\varphi_r
-2\dot s_r\dot\varphi_r\ \!{\rm ctn}_\kappa s_r\cr
\vspace{-0.4cm}\cr
\ddot\theta_r=-{\rm sn}_\kappa^{-2}s_r \sin^{-2}\varphi_r\ \!\dfrac{\partial U_\kappa}{\partial\theta_r}-2\dot s_r\dot\theta_r\ \!{\rm ctn}_\kappa
s_r - 2\dot\varphi_r\dot\theta_r\cot\varphi_r, \ \ r=\overline{1,N}.
\end{cases}
\end{equation}
Given the definitions of the unified trigonometric functions at $\kappa=0$, the above system takes in Euclidean space the form
\begin{equation}\label{final}
\begin{cases}
\ddot s_r=-\dfrac{\partial U_\kappa}{\partial s_r}+
s_r(\dot\varphi_r^2+\dot\theta_r^2\ \!\sin^2\varphi_r)\cr
\vspace{-0.4cm}\cr
\ddot\varphi_r=-s_r^{-2}\ \!\dfrac{\partial U_\kappa}{\partial\varphi_r}+\theta_r^2 \sin\varphi_r\cos\varphi_r
-2s_r^{-1}\dot s_r\dot\varphi_r\cr
\vspace{-0.4cm}\cr
\ddot\theta_r=-s_r^{-2} \sin^{-2}\varphi_r\ \!\dfrac{\partial U_\kappa}{\partial\theta_r}-2s_r^{-1}\dot s_r\dot\theta_r - 2\dot\varphi_r\dot\theta_r\cot\varphi_r, \ \ r=\overline{1,N}.
\end{cases}
\end{equation}
Since the relations in \eqref{spherical-transf} provide us with
the spherical coordinates in $\mathbb R^3$, for $\kappa=0$ we can write that
$$
x_r=s_r\sin\varphi_r\sin\theta_r,\ \
y_r=s_r\sin\varphi_r\cos\theta_r, \ \
z_r=s_r\cos\varphi_r, \ \ r=\overline{1,N}.
$$
Then straightforward computations make system \eqref{final} take the classical form \eqref{Newton3}. 
This remark completes the proof.
\end{proof}

%%%%%%%%%%%%
%%%%%%%%%%%%
%%%%%%%%%%%%
\section{Application to the gravitational $N$-body problem}
%%%%%%%%%%%%
%%%%%%%%%%%%
%%%%%%%%%%%%

In this section we will apply Theorem 2 to the gravitational $N$-body problem. As we mentioned in the introduction, the potential we choose to define on spheres and hyperbolic spheres, which has a long history, provides the natural extension of gravitation to spaces of constant curvature. The application of Theorem 1 would follow in the same way. So we consider the cotangent potential function
given by
\begin{equation}\label{N-pot}
U_\kappa({\bf q})=-\sum_{1\le i<j\le N}m_im_j{\rm ctn}_\kappa
(d_\kappa(m_i,m_j)),\ \ \kappa\ne 0,
\end{equation}
where $d_\kappa(m_i,m_j)$ represents the geodesic distance between the bodies $m_i$ and $m_j$, defined as
$$
d_\kappa(m_i,m_j)=|\kappa|^{-1/2}{\rm csn}_\kappa^{-1}
(\kappa{\bf q}_i\cdot{\bf q}_j),\ \ \kappa\ne 0,
$$
with ${\bf q}_i=(x_i,y_i,z_i,w_i)$ representing the position vector of the body $m_i,\ i=\overline{1,N}$, in a framework having the origin at the centre of the spheres. Obviously, the Euclidean distance is then
$$
d_0(m_i,m_j)=|{\bf q}_i-{\bf q}_j|.
$$
Notice that, for now, the position where the origin of the coordinate system is placed is irrelevant. Indeed, $d_\kappa(m_i,m_j)$ does not depend on this choice. However,
the choice of the origin will become relevant later, when we take the limit $\kappa\to 0$.

Straightforward computations that use the definitions of
${\rm sn}_\kappa, {\rm csn}_\kappa,$ and ${\rm ctn}_\kappa$
transform the potential \eqref{N-pot} into
\begin{equation}
U_\kappa({\bf q})=-\sum_{1\le i<j\le N}\frac{m_im_j|\kappa|^{1/2}\kappa q^{ij}}{|(\kappa q_i^2)(\kappa q_j^2)-(\kappa q^{ij})^2|^{1/2}},
\end{equation}
where we denoted
$$
q^{ij}={\bf q}_i\cdot{\bf q}_j, \ i,j=\overline{1,N},\ i\ne j, \ \ q_i^2={\bf q}_i\cdot{\bf q}_i, \ i=\overline{1,N}.
$$ 
Notice that $U_\kappa$ is defined on $\mathbb R^{4N}$,
except at the points where at least one of the denominators cancels. If we restrict the configuration space to the
manifolds $\mathbb M_\kappa^3$, i.e.\ $\mathbb S_\kappa^3$
or $\mathbb H_\kappa^3$, then 
$$
U_\kappa\colon ({\mathbb M}_\kappa^3)^N\setminus\Delta_\kappa\to(0,\infty),
$$
where $\Delta_\kappa$ represents the set of collisions and,
in case of spheres, the antipodal points as well. We would like to show now that 
\begin{equation}
\lim_{\kappa\to 0}U_\kappa({\bf q})=U_0({\bf q}):=-\sum_{1\le i<j\le N}\frac{m_im_j}{|{\bf q}_i-{\bf q}_j|},
\end{equation}
while the Euclidean distances in $\mathbb R^4$ (and the Minkowski distances in $\mathbb R^{3,1}$) between the North Pole and the bodies are fixed when $\kappa$ varies. For this let us denote
$$
q_{ij}=[(x_i-x_j)^2+(y_i-y_j)^2+(z_i-z_j)^2+\sigma(w_i-w_j)^2]^{1/2},
$$
which is the Euclidean distance in $\mathbb R^4$ for $\kappa>0$ and the Minkowski distance in $\mathbb R^{3,1}$ for $\kappa<0$. We can now proceed as in \cite{Diacu05}, where we considered the $N$-body problem in the context of curved space. So notice that
$$
2q^{ij}=q_i^2+q_j^2-q_{ij}^2.
$$
Then the potential function can be written in the ambient space as
\begin{equation}\label{gen-forcef}
U_\kappa({\bf q})=-\sum_{1\le i<j\le N}\frac{m_im_j(\kappa q_i^2+\kappa q_j^2-\kappa q_{ij}^2)}{[2(\kappa q_i^2+\kappa q_j^2)q_{ij}^2-\kappa(q_i^2-q_j^2)^2-\kappa q_{ij}^4]^{1/2}}.
\end{equation}
On the manifolds of constant curvature $\kappa$,  $U_\kappa$ becomes
\begin{equation}
U_\kappa({\bf q})=-\sum_{1\le i<j\le N}\frac{m_im_j(2-\kappa q_{ij}^2)}{q_{ij}(4-\kappa q_{ij}^2)^{1/2}},
\end{equation}
since $\kappa q_i^2=1,\  i=\overline{1,N}$. To apply Theorem 2, we must shift the origin of the coordinate system to the North Pole of all spheres, which
is also common for the vertices of all hyperbolic spheres.
The form of $U_\kappa$ does not change when we perform this transformation of coordinates because the expression of $U_\kappa$ depends only on the mutual (Euclidean/Minkowski) distances. If we keep these distances between the North Pole and the bodies fixed as $\kappa\to 0$, the mutual Euclidean/Minkowski distances between bodies remain fixed as well, so $\kappa q_{ij}\to 0$, and consequently $U_\kappa\to U_0$. By Theorem 2, the vector field defined by this potential for $\kappa\ne 0$ tends to the Newtonian vector field as $\kappa\to 0$.

\bigskip
\noindent{\bf Acknowledgment.} The authors are indebted to NSERC of Canada for partial financial support through its Discovery Grants programme.

\end{document}